\documentclass[reqno]{amsart}
\usepackage{setspace,tikz,xcolor,mathrsfs,listings,multicol}
\usepackage{amssymb}
\usepackage{rotating}
\usepackage[centertableaux]{ytableau}
\usepackage{enumerate}
\usepackage[all,cmtip]{xy}
\usetikzlibrary{arrows,matrix,positioning}
\usepackage{soul}
\tikzset{>=stealth',
  head/.style = {fill = white, text=black},
  plaque/.style = {draw, rectangle, minimum size = 10mm}, 
  pil/.style={->,thick},
  junct/.style = {draw,circle,inner sep=0.5pt,outer sep=0pt, fill=black}
  }

\usepackage[colorlinks=true, pdfstartview=FitV, linkcolor=blue, citecolor=blue, urlcolor=blue]{hyperref}

\newcommand{\G}{\mathfrak{G}}
\newcommand{\xx}{\mathbf{x}}

\newcommand{\iso}{\cong}
\newcommand{\bbb}{\mathsf{b}}

\newcommand{\fsl}{\mathfrak{sl}}
\newcommand{\sym}{S}

\DeclareMathOperator{\Stab}{Stab} 
\DeclareMathOperator{\wt}{wt} 
\DeclareMathOperator{\excess}{ex} 
\DeclareMathOperator{\ch}{ch} 
\DeclareMathOperator{\svt}{SV} 
\DeclareMathOperator{\skyline}{SLT} 
\DeclareMathOperator{\Gr}{Gr} 

\newcommand{\ZZ}{\mathbb{Z}}

\newcommand{\CC}{\mathbb{C}}

\newcommand{\mcC}{\mathcal{C}}
\newcommand{\mcD}{\mathcal{D}}
\newcommand{\mcK}{\mathcal{K}}

\newcommand{\bplus}{{\color{blue}+}}
\newcommand{\bminus}{{\color{red}-}}

\usepackage{listings}
\lstdefinelanguage{Sage}[]{Python}
{morekeywords={False,sage,True},sensitive=true}
\definecolor{dblackcolor}{rgb}{0.0,0.0,0.0}
\definecolor{dbluecolor}{rgb}{0.01,0.02,0.7}
\definecolor{dgreencolor}{rgb}{0.2,0.4,0.0}
\definecolor{dgraycolor}{rgb}{0.30,0.3,0.30}

\usepackage{xparse}

\makeatletter
\protected\def\specialmergetwolists{%
  \begingroup
  \@ifstar{\def\cnta{1}\@specialmergetwolists}
    {\def\cnta{0}\@specialmergetwolists}%
}
\def\@specialmergetwolists#1#2#3#4{%
  \def\tempa##1##2{%
    \edef##2{%
      \ifnum\cnta=\@ne\else\expandafter\@firstoftwo\fi
      \unexpanded\expandafter{##1}%
    }%
  }%
  \tempa{#2}\tempb\tempa{#3}\tempa
  \def\cnta{0}\def#4{}%
  \foreach \x in \tempb{%
    \xdef\cnta{\the\numexpr\cnta+1}%
    \gdef\cntb{0}%
    \foreach \y in \tempa{%
      \xdef\cntb{\the\numexpr\cntb+1}%
      \ifnum\cntb=\cnta\relax
        \xdef#4{#4\ifx#4\empty\else,\fi\x#1\y}%
        \breakforeach
      \fi
    }%
  }%
  \endgroup
}
\makeatother

\usepackage{xparse}
\DeclareDocumentCommand\rpp{ m m g }{
	\foreach \x [count=\s from 1] in {#1}{
	        {\ifnum\s=1
	                \draw (0,-\s)--(\x,-\s);
	                \fi}
	   \draw (0,-\s-1) to (\x,-\s-1);
	   \foreach \y in {0, ..., \x} {\draw (\y,-\s)--(\y,-\s-1);}
	}
	\specialmergetwolists{/}{#1}{#2}\ziplist
	\foreach \x/\y [count=\yi from 1] in \ziplist{
	    \node[anchor=west,font=\small] at (\x,-\yi - .5) {$\y$};
	}
	\IfValueT {#3}
	{\foreach \z [count=\zi from 1] in {#3} {\node[anchor=east,font=\small] at (0,-\zi - .5) {$\z$};}}
	{}
}

\definecolor{darkred}{rgb}{0.7,0,0} 
\newcommand{\defn}[1]{{\color{darkred}\emph{#1}}} 

\usepackage[colorinlistoftodos]{todonotes}

\theoremstyle{plain}
\newtheorem{thm}{Theorem}[section]
\newtheorem{lemma}[thm]{Lemma}
\newtheorem{conj}[thm]{Conjecture}
\newtheorem{prop}[thm]{Proposition}
\newtheorem{cor}[thm]{Corollary}
\theoremstyle{definition}
\newtheorem{dfn}[thm]{Definition}
\newtheorem{ex}[thm]{Example}

\newtheorem{remark}[thm]{Remark}
\newtheorem{prob}[thm]{Open Problem}
\numberwithin{equation}{section}

\begin{document}
\title[Krystals for rectangular shapes]{K-theoretic crystals for set-valued tableaux of rectangular shapes}

\author[O.~Pechenik]{Oliver Pechenik}
\address[O.~Pechenik]{Department of Mathematics, University of Michigan, Ann Arbor, MI 48109, USA}
\curraddr{Department of Combinatorics \& Optimization, University of Waterloo, Waterloo, ON N2L 3G1 (Canada)}
\email{oliver.pechenik@uwaterloo.ca}
\urladdr{http://www.math.uwaterloo.ca/~opecheni/}

\author[T.~Scrimshaw]{Travis Scrimshaw}
\address[T.~Scrimshaw]{School of Mathematics and Physics, The University of Queensland, St.\ Lucia, QLD 4072, Australia}
\curraddr{Osaka City University Advanced Mathematical Institute, 
Osaka City University, 
3-3-138 Sugimoto, Sumiyoshi-ku Osaka 558-8585 (Japan)}
\email{tcscrims@gmail.com}
\urladdr{https://tscrim.github.io/}


\thanks{OP was partially supported by the National Science Foundation Mathematical Sciences Postdoctoral Research Fellowship \#1703696. OP also acknowledges support from NSERC Discovery Grant RGPIN-2021-02391 and Launch Supplement DGECR-2021-00010.
TS was partially supported by the Australian Research Council DP170102648.}

\begin{abstract}
In earlier work with C.~Monical, we introduced the notion of a K-crystal, with applications to K-theoretic Schubert calculus and the study of Lascoux polynomials.
We conjectured that such a K-crystal structure existed on the set of semistandard set-valued tableaux of any fixed rectangular shape. Here, we establish this conjecture by explicitly constructing the K-crystal operators. As a consequence, we establish the first combinatorial formula for Lascoux polynomials $L_{w\lambda}$ when $\lambda$ is a multiple of a fundamental weight as the sum over flagged set-valued tableaux. Using this result, we then prove corresponding cases of conjectures of Ross--Yong (2015) and Monical (2016) by constructing bijections with the respective combinatorial objects.
\end{abstract}

\maketitle

\section{Introduction}
\label{sec:introduction}

In classical Schubert calculus, we can study the cohomology ring of the Grassmannian $\Gr(k, n)$, the parameter space for $k$-dimensional subspaces of $\CC^n$, with respect to the basis given by the Poincar\'e duals of the Schubert varieties $X_{\lambda}$ that decompose $\Gr(k, n)$. In this context, the cohomology classes $[X_{\lambda}]$ can be represented by Schur polynomials $s_{\lambda}$, where the partition $\lambda$ sits inside a $k \times (n-k)$ rectangle. A more modern approach is to study $\Gr(k, n)$ via connective K-theory, where the Schubert class $[X_{\lambda}]$ is given as the push-forward of the class for any Bott--Samelson resolution of $X_{\lambda}$. Here, polynomial representatives are given by symmetric (or stable) $\beta$-Grothendieck polynomials~\cite{FK94,Hudson}.

The Schur polynomial $s_{\lambda}$ can be described combinatorially as a generating function for semistandard (Young) tableaux of shape $\lambda$ (see, \textit{e.g.},~\cite[Ch.~7]{ECII}). In addition, $s_{\lambda}$ has a representation-theoretic interpretation as the character of the highest weight representation $V(\lambda)$ of the Lie algebra $\fsl_n$ of traceless $n \times n$ matrices (see, \textit{e.g.},~\cite[Ch.~8]{Fulton}). One way to compute $s_{\lambda}$ is by applying a product of Demazure operators $\pi_{w_0}$ corresponding to the reverse permutation $w_0$ to the monomial $\xx^{\lambda} := x_1^{\lambda_1} \dotsm x_n^{\lambda_n}$. Generalizing this formula refines Schur polynomials to the key polynomials $\kappa_{w\lambda} := \pi_w \xx^{\lambda}$, which may be understood as characters of a Demazure modules $V_w(\lambda)$~\cite{Demazure74} (hence, $\kappa_{w\lambda}$ is also known as a Demazure character); geometrically, $V_m(\lambda)$ may be constructed as global sections of a line bundle on a flag variety~\cite{Andersen85,LMS79}.

For the symmetric Grothendieck polynomial $\G_{\lambda}$, A.~Buch~\cite{Buch02} gave a combinatorial interpretation as the generating function for semistandard set-valued tableaux of shape $\lambda$. A.~Lascoux~\cite{Lascoux01} gave a deformation of the Demazure operators, called Demazure--Lascoux operators $\varpi_w$, such that $\G_{\lambda} = \varpi_{w_0} \xx^{\lambda}$. The analogous deformation of key polynomials, the so-called Lascoux polynomials $L_{w\lambda} = \varpi_w \xx^{\lambda}$, remain mysterious as currently there is no known geometric, representation-theoretic, or combinatorial interpretation, despite recent attention~\cite{RY15,Kirillov:notes,Monical16,MPS18II}. Yet, combinatorial formulas for Lascoux polynomials have been conjectured by C.~Monical~\cite[Conj.~5.3]{Monical16} and by C.~Ross and A.~Yong~\cite[Conj.~1.4]{RY15}  (with the generic $\beta$ version by A.~Kirillov~\cite[Fn.~14]{Kirillov:notes}).

One way to connect the combinatorial and representation-theoretic interpretations of key polynomials is through M.~Kashiwara's theory of crystal bases for representations of quantum groups~\cite{K90,K91}. Indeed, Kashiwara showed that the Demazure module $V_w(\lambda)$ has a crystal basis and could be described as a subcrystal $B_w(\lambda)$, called a Demazure crystal, of the highest weight crystal $B(\lambda)$~\cite{K93,L95-3}. For $U_q(\fsl_n)$, the crystal $B(\lambda)$ may be realized as the set of semistandard tableaux of shape $\lambda$, and the tableaux for the subcrystal $B_w(\lambda)$ are characterized by a combinatorial condition on their corresponding key tableaux~\cite{LS90}.

In our previous paper with C.~Monical~\cite{MPS18}, we initiated an analogous approach to Demazure crystals for Lascoux polynomials. We first gave a $U_q(\fsl_n)$-crystal structure to the set of semistandard set-valued tableaux. Then we proposed an enriched crystal structure with the property that the Lascoux polynomials appear as the characters of our K-theoretic analogs of Demazure subcrystals. We coined this enriched structure a K-crystal. We established the existence of K-crystals for single rows and columns, but we discovered that no such structure exists for general shapes.  Nonetheless, we conjectured~\cite[Conj.~7.12]{MPS18} that K-crystals exist for all rectangular shapes. Our first main result is a proof of this conjecture. Our proof gives rise to a combinatorial formula for the class of Lascoux polynomials indexed by a weight in the Weyl group orbit of a multiple of a fundamental weight \textit{i.e.}, a rectangular shape partition).  We then use this formula to establish the corresponding cases of the Ross--Yong--Kirillov and Monical conjectures. To our knowledge, these are the only proven combinatorial formulas for any class of Lascoux polynomials.\footnote{After this paper appeared as a preprint, subsequent work of V.~Buciumas, the second author, and K.~Weber~\cite{BSW20} gave another combinatorial formula for Lascoux polynomials. In particular, they proved the formula proposed here in our Conjecture~\ref{conj:Key}.}

Let us remark on why our proposed K-crystal structure exists for a rectangular shape $\lambda$, but not for general shapes. In our work with C.~Monical~\cite{MPS18}, we proposed a slightly weaker structure for general $\lambda$ that depends on a choice of a reduced expression for $w_0$. The key distinction appears to be that, in the rectangular case, the minimal-length coset representatives (such as the relevant parabolic $w_0$) that index Lascoux polynomials are all fully-commutative \textit{i.e.}, all reduced expressions differ only by commutations)~\cite{Stembridge96}. However, for more general shapes, such as $\lambda = (2,1)$ described in~\cite[Fig.~6,7]{MPS18}, one needs to apply the braid relations $s_i s_{i+1} s_i = s_{i+1} s_i s_{i+1}$ to get all possible reduced expressions. Subsequently, we believe that, in general, K-crystal structures depend on choosing a commutation class of the reduced words for the appropriate parabolic $w_0$ (see also~\cite[\S7.3]{MPS18}). This fact seems related to an analogous dependence for Schubert classes in cohomology theories more general than connective K-theory (see, \textit{e.g.},~\cite{BE90,GR13,LZ17}). Moreover, in the rectangular case, we have a flagging condition to characterize the tableaux in the K-Demazure crystal, and we expect an analogous key tableau condition to work for general shapes.

This paper is organized as follows.
In Section~\ref{sec:background}, we recall the necessary background.
In Section~\ref{sec:K_rectangles}, we construct a K-crystal structure on set-valued tableaux of rectangular shapes.
In Section~\ref{sec:Kohnert} (resp.~Section~\ref{sec:skyline_tableaux}), we prove the conjectural combinatorial interpretation of Lascoux polynomials for rectangular shapes due to Ross--Yong--Kirillov (resp.~Monical).
In Section~\ref{sec:Key_tableaux}, we describe our conjecture for key tableaux of set-valued tableaux and their relationship with Lascoux polynomials.

\subsection*{Acknowledgements}

OP is grateful for interesting conversations with Bob Proctor.  TS would like to thank Takeshi Ikeda, Tomoo Matsumura, and Shogo Sugimoto for stimulating discussions. The authors thank Cara Monical for useful discussions. The authors thank the referees for their valuable comments and suggestions. In particular, we thank one of the referees for the bridging modification to the K-crystal operators. This work benefited from computations using \textsc{SageMath}~\cite{sage, combinat}.

\section{Background}
\label{sec:background}

In this section, we give background for the symmetric group, the crystal structures on semistandard set-valued tableaux, Lascoux and Grothendieck polynomials, and the related conjectures.
Let $\xx = (x_1, x_2, x_3, \ldots)$ be a countable vector of commuting indeterminates.
For a tuple $\alpha = (\alpha_1, \alpha_2, \ldots)$, define $\xx^{\alpha} = x_1^{\alpha_1} x_2^{\alpha_2} \cdots$.
We use the English orientation convention for both partitions and tableaux.

\subsection{Properties of symmetric groups}

Let $\sym_n$ denote the symmetric group on $\{1, \dotsc, n\}$ with simple transpositions $\{s_i \mid 1 \leq i < n \}$, where $s_i$ interchanges $i$ and $i+1$.
Let $w_0 \in \sym_n$ be the reverse permutation $n(n-1) \cdots 2 1$.
A reduced expression for a permutation $w \in \sym_n$ is an expression for $w$ as a minimal-length product of simple transpositions.
The length of a permutation is the length of any reduced expression for it; the element $w_0$ is the element with greatest length in $\sym_n$.
We recall that (strong) Bruhat order on $\sym_n$ is defined by $v \leq w$ if there exists a reduced expression for $v$ that is a subword of a reduced expression for $w$.

Consider a partition $\lambda$ (of length at most $n$) as a word of length $n$ by appending $0$'s as necessary.
Note that $\sym_n$ has a natural action on words of length $n$, which corresponds to the natural action on $\ZZ^n$ of the Weyl group of $\fsl_n$ (which we can identify with the group of permutation matrices).
Let $\Stab_n(\lambda) = \{w \in \sym_n \mid w \lambda = \lambda\}$ denote the stabilizer of $\lambda$.
Recall that $\Stab_n(\lambda)$ is a parabolic subgroup of $\sym_n$ and that every coset in the quotient of a Coxeter group by a parabolic subgroup has a unique minimal length representative.
Thus, let $\sym_n^{\lambda}$ denote the set of minimal length coset representatives of $\sym_n / \Stab_n(\lambda)$, and for any $w \in \sym_n$, let $\lfloor w \rfloor$ denote the corresponding minimal length coset representative of $w$ in $\sym_n / \Stab_n(\lambda)$.
For more on Coxeter groups, we refer the reader to, \textit{e.g.},~\cite{BB05,Davis08,Humphreys90,Kane01}.

\subsection{Set-valued tableaux and their crystal structure}

Let $\lambda$ be a partition, which we often consider as a Young diagram. 
A \defn{(semistandard) set-valued tableau of shape $\lambda$} is a filling $T$ of the boxes of $\lambda$ by finite nonempty sets of positive integers so that for every set $A$ to the left of a set $B$ in the same row, we have $\max A \leq \min B$, and for $C$ below $A$ in the same column, we have $\max A < \min C$. (This is a set-valued generalization of the usual semistandard condition on tableaux.)
For an integer $a$, we write $a \in T$ if there exists a box of $T$ containing a set $A$ with $a \in A$.
A semistandard set-valued tableau is a \defn{semistandard Young tableau} if all sets have size $1$.
Let $\svt^n(\lambda)$ denote the set of all set-valued tableaux of shape $\lambda$ with entries at most $n$.

Next we recall the crystal structure on $\svt^n(\lambda)$ from~\cite{MPS18}. We first recall the \defn{crystal operators} $e_i, f_i \colon \svt^n(\lambda) \to \svt^n(\lambda) \sqcup \{0\}$, where $i \in I := \{1, \dotsc, n-1\}$. We draw the crystals as a directed graph, where we have an $i$-colored edge $T \xrightarrow{i} U$ if and only if $f_i(T) = U$. For more details on crystals, we refer the reader to~\cite{BS17,K91}.

The crystal operator $f_i$ acts on $T \in \svt^n(\lambda)$ as follows: Write $\bplus$ above each column of $T$ containing $i$ but not $i+1$, and write $\bminus$ above each column containing $i+1$ but not $i$. Now cancel signs in ordered pairs $\bminus \bplus$. If every $\bplus$ thereby cancels, then $f_i T = 0$. Otherwise let $\bbb$ correspond to the box of the rightmost uncanceled $\bplus$. Then $f_i T$ is given by one of the following:
\begin{itemize}
\item if there exists an adjacent box $\bbb^{\rightarrow}$ immediately to the right of $\bbb$ that contains an $i$, then remove the $i$ from $\bbb^{\rightarrow}$ and add an $i+1$ to $\bbb$;
\item otherwise replace the $i$ in $\bbb$ with an $i+1$.
\end{itemize}

The action of $e_i$ is defined as follows: Construct the sequence $\bplus \cdots \bplus \bminus \cdots \bminus$ as above. If there is not an uncanceled $\bminus$, then $e_i T = 0$. Otherwise let $\bbb$ correspond to the box of the leftmost uncanceled $\bminus$. Then $e_i T$ is given by one of the following:
\begin{itemize}
\item if there exists an adjacent box $\bbb^{\leftarrow}$ immediately to the left of $\bbb$ that contains an $i+1$, then remove the $i+1$ from $\bbb^{\leftarrow}$ and add an $i$ to $\bbb$;
\item otherwise replace the $i+1$ in $\bbb$ with an $i$.
\end{itemize}

Identifying $\ZZ^n$ with the multiplicative group generated by $(x_1, \dotsc, x_n)$, we define the weight function $\wt \colon \svt^n(\lambda) \to \ZZ^n$ by $\wt(T) = \prod_{i=1}^n x_i^{c_i}$, where $c_i$ is the number of $A \in T$ such that $i \in A$.
Define $|\wt(T)| = \sum_{i=1}^n c_i$.
Let $U_q(\fsl_n)$ denote the Drinfel'd--Jimbo quantum group of the type $A_{n-1}$ Lie algebra $\fsl_n$, the Lie algebra of traceless $n \times n$ matrices over $\CC$. Let $B(\lambda)$ be the highest weight $U_q(\fsl_n)$-crystal of all semistandard Young tableaux of shape $\lambda$~\cite{K90,K91,KN94}.

\begin{thm}[{\cite[Thm.~3.9]{MPS18}}]
Let $\lambda$ be a partition. Then
\[
\svt^n(\lambda) \iso \bigoplus_{\lambda \subseteq \mu} B(\mu)^{\oplus M_{\lambda}^{\mu}},
\]
where the $M_{\lambda}^{\mu} = \lvert \{ T \in \svt^n(\lambda) \mid \wt(T) = \mu \text{ and } e_i T = 0 \text{ for all } i \in I\} \rvert$.
\end{thm}

See Figure~\ref{fig:svt_crystal_ex} for an example.

\begin{figure}
\[
\ytableausetup{boxsize=2.0em}
\begin{tikzpicture}[>=latex,scale=2,every node/.style={scale=0.7}]
\node (m) at (0,0) {$\ytableaushort{11,22}$};
\node (f2m) at (0,-1) {$\ytableaushort{11,23}$};
\node (f22m) at (1,-2) {$\ytableaushort{11,33}$};
\node (f12m) at (-1,-2) {$\ytableaushort{12,23}$};
\node (f122m) at (0,-3) {$\ytableaushort{12,33}$};
\node (f1122m) at (0,-4) {$\ytableaushort{22,33}$};
\draw[->,red] (m) -- node[midway,right] {\small $2$} (f2m);
\draw[->,red] (f2m) -- node[midway,above right] {\small $2$} (f22m);
\draw[->,blue] (f2m) -- node[midway, above left] {\small $1$} (f12m);
\draw[->,red] (f12m) -- node[midway,below left] {\small $2$} (f122m);
\draw[->,blue] (f22m) -- node[midway, below right] {\small $1$} (f122m);
\draw[->,blue] (f122m) -- node[midway, right] {\small $1$} (f1122m);
\node (t) at (2,0) {$\ytableaushort{11,2{2,\!3}}$};
\node (f2t) at (2,-1) {$\ytableaushort{11,{2,\!3}3}$};
\node (f12t) at (2,-2) {$\ytableaushort{12,{2,\!3}3}$};
\draw[->,red] (t) -- node[midway,right] {\small $2$} (f2t);
\draw[->,blue] (f2t) -- node[midway, right] {\small $1$} (f12t);
\node (b) at (3,0) {$\ytableaushort{1{1,\!2},23}$};
\node (f2b) at (3,-1) {$\ytableaushort{1{1,\!2},33}$};
\node (f12b) at (3,-2) {$\ytableaushort{{1,\!2}2,33}$};
\draw[->,red] (b) -- node[midway,right] {\small $2$} (f2b);
\draw[->,blue] (f2b) -- node[midway, right] {\small $1$} (f12b);
\node (all) at (4,0) {$\ytableaushort{1{1,\!2},{2,\!3}3}$};
\end{tikzpicture}
\]
\ytableausetup{boxsize=0.8em}
\caption{The $U_q(\fsl_3)$-crystal structure on $\svt^3(2,2)$.}
\label{fig:svt_crystal_ex}
\end{figure}

\subsection{Lascoux polynomials and symmetric Grothendieck polynomials}

For $1 \leq i < n$, the \defn{Demazure operator} $\pi_i$  acts on $\ZZ[\beta][x_1, \dotsc, x_n]$ by 
\[
\pi_i f = \frac{x_i \cdot f - x_{i+1} \cdot s_i f}{x_i - x_{i+1}},
\]
where $s_i f(\ldots, x_i, x_{i+1}, \ldots) = f(\ldots x_{i+1}, x_i, \ldots)$,
and the \defn{Demazure--Lascoux operator} $\varpi_i$ acts by
\[
\varpi_i f
= \pi_i\bigl( (1  + \beta x_{i+1}) \cdot f \bigr) = \pi_i f + \beta \cdot \pi_i (x_{i+1} \cdot f).
\]
The Demazure--Lascoux operators (and Demazure operators) are known to satisfy the braid relations:
\begin{align*}
\varpi_i \varpi_j &= \varpi_j \varpi_i \hspace{1.5cm} \text{for $|i-j| > 1$}, \\
\varpi_i \varpi_{i+1} \varpi_i &= \varpi_{i+1} \varpi_i \varpi_{i+1}
\end{align*} 
(and similarly for $\pi_i$)~\cite{Lascoux01}. Thus for any permutation $w \in \sym_n$, one may unambiguously define $\varpi_w := \varpi_{i_1} \varpi_{i_2} \dotsm \varpi_{i_{\ell}}$, where $s_{i_1} s_{i_2} \dotsm s_{i_{\ell}}$ is some reduced expression for $w$.

Since $\varpi_w$ does not depend on the choice of reduced expression, we can define the \defn{Lascoux polynomials}~\cite{Lascoux01} as
\[
L_a(\xx; \beta) := \varpi_w \xx^{\lambda}
\]
for any $a \in \ZZ_{\geq 0}^n$, where $\lambda$ is the sorting of $a$ into a partition and $w \in \sym_n^{\lambda}$ is the unique element such that $w \lambda = a$.
The \defn{symmetric Grothendieck polynomial} can be defined as the $n$ variable truncation\footnote{The untruncated version is called a \emph{stable Grothendieck polynomial} as it is the stable limit $n \to \infty$ of the original Grothendieck polynomials (with $\beta = -1$) of A.~Lascoux--M.-P.~Sch\"{u}tzenberger~\cite{LS82,LS83}.} of $L_{w_0\lambda}(\xx; \beta)$ and is known~\cite[Thm.~3.1]{Buch02} to be given combinatorially by 
\begin{equation}
\label{eq:Grothendieck_defn}
L_{w_0\lambda}(\xx; \beta) = \sum_{T \in \svt^n(\lambda)} \wt_{\beta}(T),
\end{equation}
where
\[
\wt_{\beta}(T) := \beta^{\excess(T)} \wt(T),
\qquad\quad
\excess(T) := |\wt(T)| - |\lambda| = \sum_{A \in T} \bigl( |A| - 1 \bigr),
\]
where $A$ ranges over
 the entries of $T$.
The statistic $\excess \colon \svt^n(\lambda) \to \ZZ$ is known as \defn{excess}, and we call $\wt_{\beta}(T)$ the \defn{$\beta$-weight}.
There is currently no known geometric or representation-theoretic interpretation for general Lascoux polynomials. However, there are two conjectural combinatorial descriptions, which we now recall.

The first conjectural combinatorial rule was introduced in~\cite{RY15}. To state it, we begin by recalling the notion of a \defn{K-Kohnert diagram} to be a subset $D$ of $\ZZ^2_{>0}$, which we realize as boxes, and a subset $M \subseteq D$ of boxes that are marked. The conjectural rule to compute the Lascoux polynomial is as follows. Start with some $a = (a_1, \dotsc, a_n) \in \ZZ^n_{\geq 0}$ and draw the initial K-Kohnert diagram as a \defn{skyline diagram} by putting a box at each position $\{(i,y) \mid i \in [n], 1 \leq y \leq a_i\}$ (in Cartesian coordinates), marking no boxes. Then we successively apply any sequence of the following operations.
\begin{description}
\item[\defn{Kohnert move}] Move any unmarked box at the top of a column into the rightmost open position to its left and in the same row such that it does not pass through a marked box.\footnote{In the published version of~\cite{RY15}, it is misstated that a Kohnert move could move the unmarked box through a marked box. See~\cite{RossREU,RY17}.}
\item[\defn{K-Kohnert move}] Perform a Kohnert move but leave a marked box behind.
\end{description}
Let $\mcD_a$ denote the resulting set of K-Kohnert diagrams obtainable from the original skyline diagram for $a$.
Define the $\beta$-weight of a K-Kohnert diagram $D$ by $\wt_{\beta}(D) = \beta^e \prod_{i=1}^n x_i^{c_i}$, where $e$ (resp.~$c_i$) is the number of marked boxes (resp.~boxes in column~$i$) in~$D$.

\begin{conj}[{\cite[Conj.~1.4]{RY15},\cite[Fn.~14]{Kirillov:notes}}]
\label{conj:K_skyline}
We have
\[
L_a(\xx; \beta) = \sum_{D \in \mcD_a}  \wt_{\beta}(D).
\]
\end{conj}

\begin{ex}
Consider $\lambda = 2^2$ being a $2 \times 2$ rectangle and $w = s_1 s_2$, so that $a = (0,2,2)$. Then the set of K-Kohnert diagrams for $a$ is
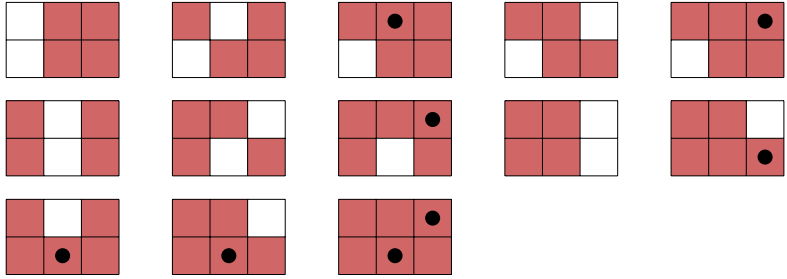
\[
\newcommand{\kohnert}[1]{\begin{tikzpicture}[scale=0.5] #1 \draw[very thin] (0,0) grid (3,2); \end{tikzpicture}}
\begin{array}{c@{\qquad}c@{\qquad}c@{\qquad}c@{\qquad}c}
\kohnert{ \fill[darkred!60] (1,0) rectangle (3,2); }
&
\kohnert{ \fill[darkred!60] (2,0) rectangle (3,2); \fill[darkred!60] (1,0) rectangle (2,1); \fill[darkred!60] (0,1) rectangle (1,2); }
&
\kohnert{ \fill[darkred!60] (1,0) rectangle (3,2); \fill[darkred!60] (0,1) rectangle (1,2); \fill[black] (1.5,1.5) circle (0.2);}
&
\kohnert{ \fill[darkred!60] (0,1) rectangle (1,2); \fill[darkred!60] (1,0) rectangle (2,2); \fill[darkred!60] (2,0) rectangle (3,1); }
&
\kohnert{ \fill[darkred!60] (0,1) rectangle (3,2); \fill[darkred!60] (1,0) rectangle (3,1); \fill[black] (2.5,1.5) circle (0.2); }
\\[5pt]
\kohnert{ \fill[darkred!60] (0,0) rectangle (1,2); \fill[darkred!60] (2,0) rectangle (3,2); } &
\kohnert{ \fill[darkred!60] (0,0) rectangle (1,2); \fill[darkred!60] (2,0) rectangle (3,1); \fill[darkred!60] (1,1) rectangle (2,2); } &
\kohnert{ \fill[darkred!60] (0,0) rectangle (1,2); \fill[darkred!60] (1,1) rectangle (2,2); \fill[darkred!60] (2,0) rectangle (3,2); \fill[black] (2.5,1.5) circle (0.2); } &
\kohnert{ \fill[darkred!60] (0,0) rectangle (1,2); \fill[darkred!60] (1,0) rectangle (2,2); } &
\kohnert{ \fill[darkred!60] (0,0) rectangle (2,2); \fill[darkred!60] (2,0) rectangle (3,1); \fill[black] (2.5,0.5) circle (0.2); }
\\[5pt]
\kohnert{ \fill[darkred!60] (0,0) rectangle (1,2); \fill[darkred!60] (2,0) rectangle (3,2); \fill[darkred!60] (1,0) rectangle (2,1); \fill[black] (1.5,0.5) circle (0.2); } &
\kohnert{ \fill[darkred!60] (0,0) rectangle (2,2); \fill[darkred!60] (2,0) rectangle (3,1); \fill[black] (1.5,0.5) circle (0.2); } &
\kohnert{ \fill[darkred!60] (0,0) rectangle (3,2); \fill[black] (1.5,0.5) circle (0.2); \fill[black] (2.5,1.5) circle (0.2);} &
\end{array}
\]
Hence, Conjecture~\ref{conj:K_skyline} (correctly) predicts that 
\begin{align*}
L_{(0,2,2)}(\xx; \beta) &= \xx^{022} + \xx^{112} + \beta \xx^{122} + \xx^{121} + \beta \xx^{122} \\
& \hspace{20pt} + \xx^{202} + \xx^{211} + \beta \xx^{212} + \xx^{220} + \beta \xx^{221} \\
& \hspace{20pt} + \beta \xx^{212} + \beta \xx^{221} + \beta^2 \xx^{222}.
\end{align*}
\end{ex}

Conjecture~\ref{conj:K_skyline} specializes at $\beta=0$ to A.~Kohnert's combinatorial formula for the monomial expansion of Demazure characters~\cite{Kohnert91}.

The second conjectural combinatorial rule is from~\cite{Monical16}. We fill a skyline diagram with finite nonempty sets of positive integers that satisfy the following conditions. Call the largest entry in a box the \defn{anchor} and the other entries \defn{free}.
\begin{enumerate}[(S.1)]
\item Entries do not repeat in a row.
\item \label{weak_increase} If $B$ is below $A$, then $\min B \geq \max A$ \textit{i.e.}, the columns are weakly increasing top-to-bottom in the set-valued sense).
\item For every triple of boxes of the form
\[
\begin{array}{c@{\hspace{40pt}}c}
\begin{array}{|c|c|c|}
\cline{1-1} \cline{3-3}
A & \cdots & C
\\ \cline{1-1} \cline{3-3}
B & \multicolumn{1}{c}{}
\\ \cline{1-1}
\end{array}
&
\begin{array}{|c|c|c|}
\cline{3-3}
\multicolumn{1}{c}{} & & A
\\ \cline{1-1} \cline{3-3}
C & \cdots & B
\\ \cline{1-1} \cline{3-3}
\end{array}
\\[10pt]
\text{left column weakly taller}
&
\text{right column strictly taller}
\end{array}
\]
the anchors $a,b,c$ of $A,B,C$, respectively, must satisfy either $c < a$ or $b < c$.\footnote{Such triples were originally called \defn{inversion triples} and required to satisfy $c < a \leq b$ or $a \leq b < c$, but in our case $a \leq b$ is immediate by~\ref{weak_increase}.}
\item Every free entry is in the leftmost cell of its row such that the entry remains free and~\ref{weak_increase} is not violated.
\item Anchors in the bottom row equal their column index.
\end{enumerate}
We call such a tableau a \defn{(semistandard) set-valued skyline tableau}.
For $a$ a weak composition (\textit{e.g.}, a finite string of nonnegative integers), let $\skyline_a$ denote the set of set-valued skyline tableaux with shape $a$.
We define the weight, excess, and $\beta$-weight for a set-valued skyline tableau in the same way as for a set-valued tableau.

Let $\overline{\varpi}_i = \varpi_i - 1$.
Define the \defn{Lascoux atom} to be
\[
\overline{L}_{w\lambda}(\xx; \beta) := \overline{\varpi}_w \xx^{\lambda}.
\]

\begin{conj}[{\cite[Conj.~5.2]{Monical16}}]
\label{conj:skyline_tableaux}
We have
\[
\overline{L}_{w\lambda} = \sum_{S \in \skyline_{w\lambda}} \wt_{\beta}(S).
\]
\end{conj}

From~\cite[Thm.~5.1]{Monical16}, we have
\begin{equation}
\label{eq:lascoux_to_atoms}
L_{w\lambda}(\xx; \beta) = \sum_{v \leq w} \overline{L}_{v \lambda}(\xx; \beta),
\end{equation}
where the sum is taken over all permutations $v$ less than or equal to $w$ in Bruhat order.
Thus, if Conjecture~\ref{conj:skyline_tableaux} holds, then we have another combinatorial interpretation of the Lascoux polynomial as
\[
L_{w\lambda}(\xx; \beta) = \sum_{v \leq w} \sum_{S \in \skyline_{v\lambda}} \wt_{\beta}(S).
\]
By~\cite{MPS18II}, this interpretation of Lascoux polynomials is equivalent to~\cite[Conj.~5.3]{Monical16} .

\subsection{K-crystals}

We recall a proposed K-theory analog of crystals that was introduced in~\cite{MPS18}. For a nilpotent operator $\psi$, we write $\psi^{\max}(x)$ to mean $\psi^m(x)$, where $m = \max \{ h  \in \mathbb{Z}_{>0} \mid \psi^h(x) \neq 0 \}$.

We call a connected $U_q(\fsl_n)$-crystal $B$ a \defn{connected K-crystal} if it is enhanced with \defn{K-crystal operators}, $e_i^K, f_i^K \colon B \to B \sqcup \{0\}$ that satisfy the following properties:
\begin{enumerate}[(K.1)]
\item \label{Krystal:connected} The set $B$ is generated by a unique element $u \in B$ that satisfies $e_i u = 0$ and $e_i^K u = 0$ for all $i \in I$. That is to say, we can reach every element in $B$ by applying a sequence of (K-)crystal operators from $u$. The element $u$ is called the \defn{minimal highest weight element}.
\item \label{Krystal:demazure} Let $w = s_{i_1} \cdots s_{i_{\ell}} \in \sym_n$ be a reduced expression. The \defn{K-Demazure crystal}
  \[
  B_w := \left\{ b \in B \mid (e_{i_{\ell}}^K)^{\max} e_{i_{\ell}}^{\max} \cdots (e_{i_1}^K)^{\max} e_{i_1}^{\max} b = u \right\}
  \]
  does not depend on the choice of reduced expression of $w$. Moreover, we have $B_{w_0} = B$.
\item \label{Krystal:character} Let $\lambda = \wt(u)$ be the weight of the minimal highest weight element from~\ref{Krystal:connected}. The \defn{$\beta$-character} of $B_w$
\[
\ch_{\beta}(B_w) := \sum_{b \in B_w} \beta^{|\wt(b)| - |\lambda|} \wt(b)
\]
is equal to the Lascoux polynomial $L_{w\lambda}(\xx; \beta)$.
\end{enumerate}
A \defn{K-crystal} is just a disjoint union of connected K-crystals. It is also strongly desirable for these operators to also satisfy $e_i^K b = b'$ if and only if $b = f_i^K b'$ for all $b, b' \in B$ and $\wt(f_i^K b) = x_{i+1} \wt(b)$.

\begin{remark}
This definition of a K-crystal is slightly more general than that given in~\cite{MPS18}, which was directly based on the combinatorics of set-valued tableaux. In particular, the extra condition that $B_{w_0} = B$ in~\ref{Krystal:demazure} is implicit in the K-crystal definition given in~\cite{MPS18} as $\ch_{\beta}(B_{w_0}) = L_{w_0 \lambda}(\xx; \beta) = \ch_{\beta}(B)$ by Equation~\eqref{eq:Grothendieck_defn}.
\end{remark}

In~\cite{MPS18}, such K-crystals were constructed for the cases that $\lambda$ is a single-row~\cite[Thm.~7.5]{MPS18} or single-column~\cite[Thm.~7.9]{MPS18}. 

\begin{prob}[{\cite[Open Prob.~7.1]{MPS18}}]
\label{prob:Krystal}
Construct an appropriate K-theory analog of crystals for general $\lambda$.
\end{prob}

It was shown in~\cite{MPS18} that there is no K-crystal (as defined here) solving Open Problem~\ref{prob:Krystal}. Rather, in general, we believe that~\ref{Krystal:demazure} should be relaxed, so that the K-crystal operators can depend on a choice of reduced word, giving a structure that we coined a weak K-crystal. Our main result is to prove~\cite[Conj.~7.12]{MPS18}, thus giving an answer to Open Problem~\ref{prob:Krystal} for rectangular shapes, extending the single row and column results of~\cite{MPS18}. In this rectangular case, the relevant Weyl group elements $\sym_n^{\lambda}$ are all fully-commutative, so the expected dependence on reduced word does not appear.

\section{K-crystals for rectangular shapes}
\label{sec:K_rectangles}

In this section, we prove our main result: when $\lambda$ is a rectangle, then $\svt^n(\lambda)$ has a K-crystal structure. Thereby, we establish~\cite[Conj.~7.12]{MPS18}, providing a solution to Open Problem~\ref{prob:Krystal} for rectangular shapes.

Our construction of the K-crystal operators is motivated by the heuristics given in~\cite{MPS18}, which come from the following K-theory analog of the decomposition of a crystal into $i$-strings \textit{i.e.}, restricting to the action of $e_i$ and $f_i$ for a fixed $i \in I$) based on the definition of the Demazure--Lusztig operators. Indeed, by considering only the action of a fixed $i \in I$, we predicted in~\cite{MPS18} that the K-crystal should decompose into (maximal) subcrystals of the form
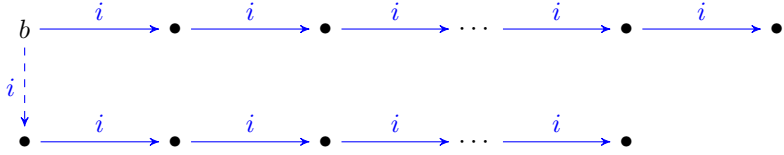
\[
\begin{tikzpicture}[xscale=2,yscale=-1.5]
\node (top) at (0,0)  {$b$};
\node (f) at (1,0)  {$\bullet$};
\node (ff) at (2,0)  {$\bullet$};
\node (dots) at (3,0) {$\cdots$};
\node (fend) at (4,0)  {$\bullet$};
\node (ffend) at (5,0)  {$\bullet$};
\node (K) at (0,1)  {$\bullet$};
\node (fK) at (1,1)  {$\bullet$};
\node (ffK) at (2,1)  {$\bullet$};
\node (dotsK) at (3,1) {$\cdots$};
\node (fendK) at (4,1)  {$\bullet$};
\draw[->,blue] (top) -- node[midway,above] {$i$} (f);
\draw[->,blue] (f) -- node[midway,above] {$i$} (ff);
\draw[->,blue] (ff) -- node[midway,above] {$i$} (dots);
\draw[->,blue] (dots) -- node[midway,above] {$i$} (fend);
\draw[->,blue] (fend) -- node[midway,above] {$i$} (ffend);
\draw[->,blue] (K) -- node[midway,above] {$i$} (fK);
\draw[->,blue] (fK) -- node[midway,above] {$i$} (ffK);
\draw[->,blue] (ffK) -- node[midway,above] {$i$} (dotsK);
\draw[->,blue] (dotsK) -- node[midway,above] {$i$} (fendK);
\draw[->,blue,dashed] (top) -- node[midway,left] {$i$} (K);
\end{tikzpicture}
\]
where the solid (resp.~dashed) arrow represents the $f_i$ (resp.~$f_i^K$) action and the top $i$-string has length one more than the bottom $i$-string.
Such a subcrystal was coined an \defn{$i$-K-string} in~\cite{MPS18}. We say an $i$-K-string has \defn{length}
\[
\ell := \max \{k \mid f_i^k b \neq 0\}.
\]
Note that $f_i^{\ell-1} f_i^K b \neq 0$ and $f_i^{\ell} f_i^K b = 0$. It is easy to see that for $b$ such that $e_i b = 0$ (which implies for $\wt(b) = x_1^{a_1} \dotsm x_n^{a_n}$ we have $a_i \geq a_{i+1}$), the $\beta$-character of the $i$-K-string starting at $b$ equals $\varpi_i \wt(b)$.

For the remainder of this section, we consider $\lambda = s^r$ to be an $r \times s$ rectangle.

The following lemma is straightforward after recalling that
\begin{align*}
\Stab_n(\lambda) & = \sym_r \times \sym_{n-r}, \text{and}
\\
\sym_n^{\lambda} & = \{ w \in \sym_n \mid w(1) < \cdots < w(r) \text{ and } w(r+1) < \cdots < w(n) \}.
\end{align*}

\begin{lemma}
\label{lemma:coset_repr}
Let $\lambda = s^r$ be an $r \times s$ rectangle.
For any $w \in \sym_n^{\lambda}$, there exists a reduced expression of $w$ of the form
\[
(s_{i_k} \dotsm s_{r-k+1} s_{r-k}) \dotsm (s_{i_1} \dotsm s_{r} s_{r-1}) (s_{i_0} \dotsm s_{r+1} s_r)
\]
for some $-1 \leq k < r$ and $1 \leq i_k < \cdots < i_1 < i_0 \leq n$. (The case $k=-1$ means there are no terms in the product, so $w = 1$.)
\end{lemma}

\begin{ex}
	To clarify our notation and symmetric group conventions, consider $w = 2357146$, so $n=7$ and $r = 4$. Then the reduced expression for $w$ in the form given by Lemma~\ref{lemma:coset_repr} is
	\[
	w = (s_1) (s_2) (s_4s_3)(s_6s_5s_4).
	\] In particular, we have $k=3$, $i_0 = 6$, $i_1 = 4$, $i_2 = 2$, and $i_3=i_k=1$.
\end{ex}

\begin{dfn}
\label{def:K_crystal_ops}
Let $T \in \svt^n(\lambda)$, and fix some $i \in I$. We say a box $\bbb$ of $T$ is \defn{bridged} (for $i$) if there exists an $i$ strictly to the right of $\bbb$ paired with an $i+1$ strictly to the left of $\bbb$. Otherwise $\bbb$ is \defn{unbridged}.
\begin{description}
\item[\defn{$f_i^K$}] If $i \notin T$ or $f_i T = 0$ or $e_i T \neq 0$, then $f_i^K T = 0$. Otherwise, let $\bbb$ be the rightmost box that contains an $i$ corresponding to an uncanceled $\bplus$.
If $i$ and $i+1$ are both in an unbridged box to the right of $\bbb$, then $f_i^K T = 0$.
Otherwise, define $f_i^K T$ by adding an $i+1$ to $\bbb$.
\item[\defn{$e_i^K$}] If there does not exist an unbridged box with both an $i$ and $i+1$ or $e_i T \neq 0$, then $e_i^K T = 0$. Otherwise, let $\bbb$ be the rightmost unbridged box that contains both an $i$ and $i+1$. If there exists an $i$ to the right of $\bbb$ corresponding to an uncanceled $\bplus$, then $e_i^K T = 0$. Otherwise, define $e_i^K T$ by removing the $i+1$ from $\bbb$.
\end{description}
\end{dfn}

For examples of these operators, see Figure~\ref{fig:Krystal_ex_22}; additional examples may be found in~\cite{MPS18}. It is clear that if $f_i^K T \neq 0$ (resp.~$e_i^K T \neq 0$), then $f_i^K T \in \svt^n(\lambda)$ (resp.~$e_i^K T \in \svt^n(\lambda)$). We give an example to illustrate the bridging condition.

\begin{ex}
We apply the K-crystal operator $e_4^K$, which acts on the left box containing $\{{\color{blue}\mathbf{4},\mathbf{5}}\}$, which is unbridged, as the right box containing $\{{\color{darkred}4,5}\}$ is bridged:
\[
\ytableausetup{boxsize=1.5em}
\ytableaushort{111114,2222{\color{darkred}4,\!5}9,33359{10},44{\color{blue}\mathbf{4},\!\mathbf{5}}9{10}{11},579{10}{11}{12}}
\quad
\xrightarrow[\hspace{30pt}]{e^K_4}
\quad
\ytableaushort{111114,2222{\color{darkred}4,\!5}9,33359{10},44{\color{blue}\mathbf{4}}9{10}{11},579{10}{11}{12}}\,.
\]
\end{ex}

\begin{lemma}
\label{lemma:inverse_operations}
Let $T, T' \in \svt^n(\lambda)$.
We have $e_i^K T' = T$ if and only if $T' = f_i^K T$.
\end{lemma}

\begin{proof}
We first show $e_i^K T' = T$ implies $T' = f_i^K T$. From our assumption, we have $e_i T' = 0$ and for the rightmost unbridged box $\bbb$ with $i, i+1 \in \bbb$ (note $\bbb$ exists by our assumption), there does not exist an $i$ to the right of $\bbb$ that is either an uncanceled $\bplus$ or one paired with an $i+1$ to the left of $\bbb$. Therefore, when we remove the $i+1$ from $\bbb$ to obtain $T$, we create an uncanceled $\bplus$ for $\bbb$. So $e_i T = 0$ and $f_i T \neq 0$. Furthermore, this added uncanceled $\bplus$ is the rightmost such uncanceled $\bplus$ and there are no other unbridged boxes to the right of $\bbb$ that contain $i,i+1$ in $T$. Hence the action of $f_i^K$ on $T$ adds $i+1$ to $\bbb$, and thus $T' = f_i^K T$.

Now we show $T' = f_i^K T$ implies $e_i^K T' = T$. From our assumption, we have $e_i T = 0$ and there does not exist an unbridged box of $T$ containing both $i$ and $i+1$ to the right of the rightmost box $\bbb$ corresponding to a uncanceled $\bplus$. 
Since the $\bplus$ in $\bbb$ was uncanceled, removing it to form $T'$ does not affect any other cancelations. Hence, it follows that $e_i T' = 0$ as well since $e_i T = 0$.
The effect of adding an $i+1$ to $\bbb$ cancels the $\bplus$ corresponding to $\bbb$. Since $\bbb$ corresponded to the rightmost uncanceled $\bplus$, there is no uncanceled $\bplus$ to the right of $\bbb$ in $T$. Moreover, there is not an $i$ to the right of $\bbb$ that pairs with an $i+1$ to the left of $\bbb$ as otherwise such an $i+1$ would pair with the $i$ in $\bbb$. Hence $e_i^K$ acts on $T'$ by removing $i+1$ from $\bbb$, and thus $e_i^K T' = T$.
\end{proof}

\begin{figure}
\[
\ytableausetup{boxsize=2.0em}
\begin{tikzpicture}[>=latex,scale=2,every node/.style={scale=0.7},baseline=0]
\node (m) at (0,0) {$\ytableaushort[*(lightgray)]{11,22}$};
\node (f2m) at (0,-1) {$\ytableaushort[*(lightgray)]{11,23}$};
\node (f22m) at (1,-2) {$\ytableaushort[*(lightgray)]{11,33}$};
\node (f12m) at (-1,-2) {$\ytableaushort{12,23}$};
\node (f122m) at (0,-3) {$\ytableaushort{12,33}$};
\node (f1122m) at (0,-4) {$\ytableaushort{22,33}$};
\draw[->,red] (m) -- node[midway,right] {\small $2$} (f2m);
\draw[->,red] (f2m) -- node[midway,above right] {\small $2$} (f22m);
\draw[->,blue] (f2m) -- node[midway, above left] {\small $1$} (f12m);
\draw[->,red] (f12m) -- node[midway,below left] {\small $2$} (f122m);
\draw[->,blue] (f22m) -- node[midway, below right] {\small $1$} (f122m);
\draw[->,blue] (f122m) -- node[midway, right] {\small $1$} (f1122m);
\node (t) at (-2,0) {$\ytableaushort[*(lightgray)]{11,2{2,\!3}}$};
\node (f2t) at (-2,-1) {$\ytableaushort[*(lightgray)]{11,{2,\!3}3}$};
\node (f12t) at (-2,-2) {$\ytableaushort{12,{2,\!3}3}$};
\draw[->,red] (t) -- node[midway,right] {\small $2$} (f2t);
\draw[->,blue] (f2t) -- node[midway, right] {\small $1$} (f12t);
\node (b) at (2,-1) {$\ytableaushort{1{1,\!2},23}$};
\node (f2b) at (2,-2) {$\ytableaushort{1{1,\!2},33}$};
\node (f12b) at (2,-3) {$\ytableaushort{{1,\!2}2,33}$};
\draw[->,red] (b) -- node[midway,right] {\small $2$} (f2b);
\draw[->,blue] (f2b) -- node[midway, right] {\small $1$} (f12b);
\node (all) at (3,-1) {$\ytableaushort{1{1,\!2},{2,\!3}3}$};
\draw[->,dashed,red] (m) -- node[midway,above] {\small $2$} (t);
\draw[->,dashed,blue] (f2m) -- node[midway,above] {\small $1$} (b);
\draw[->,dashed,blue] (f22m) -- node[midway,above] {\small $1$} (f2b);
\draw[->,dashed,red] (f12m) -- node[midway,above] {\small $2$} (f12t);
\draw[->,dashed,red] (b) -- node[midway,above] {\small $2$} (all);
\draw[->,dashed,blue] (f2t) .. controls (-0.5,-0.2) and (2.2,-0.2) .. node[midway,above] {\small $1$} (all);
\end{tikzpicture}
\]
\ytableausetup{boxsize=0.8em}
\caption{The K-crystal for $\svt^3\left(\ydiagram{2,2} \right)$ with the K-crystal operators depicted by dashed lines. The K-Demazure crystal $B_{s_2}$ is the restriction to the subset of shaded tableaux.}
\label{fig:Krystal_ex_22}
\end{figure}

\begin{lemma}
\label{lemma:K_strings}
Let $\lambda = s^r$ be an $r \times s$ rectangle.
The restriction to any fixed $i \in I$ decomposes $\svt^n(\lambda)$ into $i$-K-strings.
\end{lemma}

\begin{proof}
Let $T \in \svt^n(\lambda)$.
Since $f_i^K T = 0$ and $e_i^K T = 0$ whenever $e_i T \neq 0$, we cannot have the local situations around $T$
\[
\begin{tikzpicture}[xscale=2,yscale=-1.5]
\node (top) at (0,0)  {$T$};
\node (f) at (1,0)  {$\cdots$};
\node (K) at (0,1)  {$\bullet$};
\node (fK) at (1,1)  {$\cdots$};
\node (prev) at (-1,0) {$\cdots$};
\draw[->,blue] (top) -- node[midway,above] {$i$} (f);
\draw[->,blue] (K) -- node[midway,above] {$i$} (fK);
\draw[->,blue] (prev) -- node[midway,above] {$i$} (top);
\draw[->,blue,dashed] (top) -- node[midway,left] {$i$} (K);
\end{tikzpicture}
\qquad\qquad
\begin{tikzpicture}[xscale=2,yscale=-1.5]
\node (top) at (0,0)  {$T$};
\node (f) at (1,0)  {$\cdots$};
\node (K) at (0,1)  {$\bullet$};
\node (fK) at (1,1)  {$\cdots$};
\node (prev) at (-1,1) {$\cdots$};
\draw[->,blue] (top) -- node[midway,above] {$i$} (f);
\draw[->,blue] (K) -- node[midway,above] {$i$} (fK);
\draw[->,blue] (prev) -- node[midway,above] {$i$} (K);
\draw[->,blue,dashed] (top) -- node[midway,left] {$i$} (K);
\end{tikzpicture}
\]
respectively.
Next if $f_i^K T \neq 0$, then we have $f_i^K f_i^K T = 0$ from the fact that we added an $i+1$ to the box $\bbb$ corresponding to the rightmost uncanceled $\bplus$, which means the rightmost uncanceled $\bplus$ in $f_i^K T$ is to the left of $\bbb$, which is necessarily unbridged in $f_i^K T$ as otherwise the $\bplus$ in $\bbb$ would be paired.

We need to show that if $e_i T = 0$ and $f_i T \neq 0$, we have either $f_i^K T \neq 0$ or $e_i^K T \neq 0$. Since $f_i T \neq 0$, there exists at least one uncanceled $\bplus$. Let $\bbb$ be the box corresponding to the rightmost uncanceled $\bplus$. If there are no unbridged boxes containing $i$ and $i+1$ to the right of $\bbb$, then we have $f_i^K T \neq 0$. Otherwise let $\bbb'$ be such an unbridged box, and since $\bbb$ contains the rightmost uncanceled $\bplus$, there are no uncanceled $\bplus$ to the right of $\bbb'$. Hence, we have $e_i^K T \neq 0$.

Now we assume $f_i^K T \neq 0$. We have $f_i^{\ell} T = 0$ if and only if $f_i^{\ell-1} f_i^K T = 0$ from the fact that to obtain $f_i^K T$, we removed the rightmost uncanceled $\bplus$ in $T$ from $\bbb$, which leaves the other uncanceled $\bplus$ and $\bminus$ unchanged.
Consequently, we have $\varphi_i(f_i^K T) = \varphi_i(T) - 1$.
Finally, we have $f_i^k T \neq f_i^m f_i^K T$ for all $0 \leq k \leq \ell$ and $0 \leq m \leq \ell - 1$ since $\excess(f_i^K T) = \excess(T) + 1$.
\end{proof}

Consider some $w \in \sym_n^{\lambda}$, and let $i_k < \cdots < i_0$ be from the reduced expression of $w$ given by Lemma~\ref{lemma:coset_repr}. For all $k < j < r$, define $i_j = r - j - 1$.
Define $F(\lambda; w)$ to be the subset of $\svt^n(\lambda)$ such that row $r-j$ has all entries at most $i_j + 1$. Equivalently, the entries in row $j$ are at most $w(j)$ for all $1 \leq j \leq r$. We call such a set-valued tableau a \defn{flagged set-valued tableau}.
Diagrammatically, the flagging in each row is given by the labels on the right
\[
\ytableausetup{boxsize=1.5em}
\ytableaushort{{\ast}{\ast}{\cdots}{\ast}{\none[1]},{\vdots}{\vdots}{\ddots}{\vdots}{\none[\vdots]},{\ast}{\ast}{\cdots}{\ast}{\none[\qquad\quad\ r-k-1]},{\ast}{\ast}{\cdots}{\ast}{\none[\qquad i_k+1]},{\vdots}{\vdots}{\ddots}{\vdots}{\none[\vdots]},{\ast}{\ast}{\cdots}{\ast}{\none[\qquad i_0+1]}}
\]

\begin{ex}
Suppose $w = s_2$ and $\lambda = 2 \times 2$. Then the flagged tableaux in $F(\lambda; w)$ are those set-valued tableaux with shape $2 \times 2$ satisfying the flagging condition that entries of the first row are bounded by $1$ and entries of the second row are bounded by $3$. These bounds can be seen from the fact that
\[
s_2 \{1, 2\} = \{s_2(1), s_2(2)\} = \{1, 3\}.
\]
The tableaux of $F(\lambda; w)$  are precisely the shaded tableaux illustrated in Figure~\ref{fig:Krystal_ex_22}.
\end{ex}

As the next lemma indicates, the flagging conditions characterize K-Demazure crystals. Recall that $\svt^n_w(\lambda)$ denotes the K-Demazure crystal of $\svt^n(\lambda)$ corresponding to $w \in \sym_n$.

\begin{lemma}
\label{lemma:flag_condition}
Let $\lambda = s^r$ be an $r \times s$ rectangle.
For $w \in \sym_n$, we have
\[
\svt^n_w(\lambda) = \svt^n_{\lfloor w \rfloor}(\lambda) = F(\lambda; \lfloor w \rfloor).
\]
\end{lemma}

\begin{proof}
Let $u$ be the minimal highest weight tableau of shape $\lambda$.
That 
\[
\svt^n_w(\lambda) = \svt^n_{\lfloor w \rfloor}(\lambda)
\] is immediate from the fact that we have $f_i u = 0$ for all $i \in I \setminus \{r\}$, \textit{i.e.}, whenever $s_i \in \Stab_n(\lambda)$. Hence, for the remainder of the proof we assume $w = \lfloor w \rfloor$.

Using Lemma~\ref{lemma:coset_repr}, write $w$ as a reduced expression
\begin{equation}\label{eq:red}
(s_{i_k} \dotsm s_{r-k+1} s_{r-k}) \dotsm (s_{i_1} \dotsm s_{r} s_{r-1}) (s_{i_0} \dotsm s_{r+1} s_r)
\end{equation}
for some $-1 \leq k < r$ and $1 \leq i_k < \cdots < i_1 < i_0 \leq n$. 

We prove the lemma by induction on $k$. The base case of $k = -1$ (so $w = 1$) is trivial since clearly
\[
\svt^n_1(\lambda) = F(\lambda;  1) = \{ u \}.
\]
Otherwise, $k\geq 0$ and let 
\[
w' = (s_{i_{k-1}} \dotsm s_{r-k+1} s_{r-(k-1)}) \dotsm (s_{i_1} \dotsm s_{r} s_{r-1}) (s_{i_0} \dotsm s_{r+1} s_r) \in \sym_n^{\lambda}
\]
be the product of all but the leftmost factor in our reduced expression~\eqref{eq:red} for $w$. Inductively, we assume 
\[
\svt^n_{w'}(\lambda) = F(\lambda;  w').
\]

By the flagging and semistandard conditions, for any tableau $T \in \svt^n_{w'}(\lambda)$ and $j \leq r - k$, the entries in row $j$ of $T$ are all exactly $j$.
Furthermore, every tableau $U \in \svt^n_w(\lambda) \setminus \svt^n_{w'}(\lambda)$ merely differs from an element in $\svt^n_{w'}(\lambda)$ by changing one or more entries in row $r - k$. 

It remains to check that the tableaux $U \in \svt^n_w(\lambda) \setminus \svt^n_{w'}(\lambda)$ differ from the elements of $\svt^n_{w'}(\lambda)$ exactly by the appropriate change in the flagging condition on row $r - k$. This check is identical to the proof of~\cite[Lemma~7.4]{MPS18} (which is the $r=1$ case of the current lemma), except for being notationally more cumbersome. It is a straightforward induction on the length $i_k - (r-k)$ of the leftmost factor of the reduced expression \eqref{eq:red}.
%
%
\end{proof}

The following theorem confirms~\cite[Conj.~7.12]{MPS18}.

\begin{thm}
\label{thm:Krystal_rectangles}
Let $\lambda = s^r$ be an $r \times s$ rectangle.
Then $\svt^n(\lambda)$ is a K-crystal.
\end{thm}

\begin{proof}
(K.1) is immediately from Lemma~\ref{lemma:flag_condition}.
(K.3) follows from Lemma~\ref{lemma:flag_condition} and that the properties of the Demazure--Lascoux operators imply $\varpi_w \xx^{\lambda} = \varpi_{\lfloor w \rfloor} \xx^{\lambda}$.
To show (K.2), note that Lemma~\ref{lemma:flag_condition} implies $\svt^n_w(\lambda)$ only depends on the minimal length coset representative.
By~\cite[Thm.~6.1]{Stembridge96}, every minimal length coset representative of $\sym_n^{\lambda}$ is fully-commutative (see also~\cite[Prop.~2.4]{Stembridge96}); in other words, they differ only by the commutation relations $s_i s_j = s_j s_i$ for $\lvert i - j \rvert > 1$.
It is clear that the (K-)crystal operators $f_i, f_i^K$ commute with $f_j, f_j^K$ for $\lvert i - j \rvert > 1$, and hence the K-Demazure crystal is independent of the choice of reduced expression.
\end{proof}

We also have the following K-theoretic analog of~\cite[Prop.~3.3.4]{K93}. 

\begin{cor}
\label{cor:demazure_strings}
Let $\lambda = s^r$ be an $r \times s$ rectangle.
Consider an $i$-K-string $S$ of $\svt^n(\lambda)$, and let $b$ be the highest weight element of $S$. Then, the set $\svt^n_w(\lambda) \cap S$ is either empty, $S$, or $\{b\}$.
\end{cor}

\begin{proof}
This follows immediately from Lemma~\ref{lemma:flag_condition}, the semistandardness, and that the flagging on the rows is strictly increasing.
\end{proof}

We also have the following interpretation of certain Lascoux polynomials as instances of (nonsymmetric) Grothendieck polynomials, indexed by some $w \in \sym_n$. Recall from~\cite{Lascoux90,LS82,LS83,FK94} that the (nonsymmetric) \defn{($\beta$-)Grothendieck polynomial} is defined by
\[
\G_{w_0 s_{i_1} \cdots s_{i_{\ell}}} := \partial^{\beta}_{i_1} \cdots \partial^{\beta}_{i_{\ell}} x_1^{n-1} \dotsm x_{n-1}^1 x_n^0,
\quad
\partial_i^{\beta} f = \frac{(1 + \beta x_{i+1}) \cdot f - (1 + \beta x_i) \cdot s_i f}{x_i - x_{i+1}},
\]
where $s_{i_1} \dotsm s_{i_{\ell}}$ is a reduced expression.

\begin{cor}
\label{cor:vexillary_Lascoux}
Let $\lambda = s^r$ be an $r \times s$ rectangle. Let
\[
w = (s_k \dotsm s_2 s_1) (s_{k+1} \dotsm s_3 s_2) \dotsm (s_{k+r-1} \dotsm s_{r+1} s_r)
\]
for some $k \geq 1$, and let
\[
\widetilde{w} = s_{m-1} (s_{m-2} s_{m-1}) \dotsm (s_{r+1} \dotsm s_{m-1}) (s_r \dotsm s_{k-1}) \dotsm (s_1 \dotsm s_{k-1})
\in S_m
\]
where $m = s + k + 1$.
Then, we have
\[
L_{w\lambda}(\xx; \beta) = \G_{w_0\widetilde{w}^{-1}}(\xx; \beta).
\]
\end{cor}

\begin{proof}
Theorem~\ref{thm:Krystal_rectangles} shows that $L_{w\lambda}(\xx; \beta)$ is the character of the K-Demazure crystal and Lemma~\ref{lemma:flag_condition} shows that this character is a generating function for a class of flagged set-valued tableaux. The identification with a Grothendieck polynomial then follows from the formula of~\cite[Thm.~3.3]{Matsumura19} (alternatively~\cite[Thm.~5.8]{KMY09}).
\end{proof}

It is clear that the permutations $w_0\widetilde{w}^{-1}$ appearing in Corollary~\ref{cor:vexillary_Lascoux} are vexillary \textit{i.e.} $2143$-avoiding). Since the greatest term of $L_{w\lambda}(\xx; 0) $ in reverse lexicographic order is $\xx^{w \lambda}$ and the greatest term of $\G_{w_0\widetilde{w}^{-1}}(\xx; 0)$ in the same order is the Lehmer code of $w_0\widetilde{w}^{-1}$, we find that $w \lambda$ is the Lehmer code of $w_0\widetilde{w}^{-1}$. Hence, the permutations $w_0\widetilde{w}^{-1}$ are Grassmannian, and so the Grothendieck polynomials appearing in Corollary~\ref{cor:vexillary_Lascoux} are actually \emph{symmetric} Grothendieck polynomials, but symmetric only in some initial segment of the variables $\xx$.

\begin{ex}
Let $\lambda = 2^2$ be a $2 \times 2$ rectangle. We have
\begin{align*}
L_{s_1 s_2 \lambda}(\xx; \beta) & = \G_{w_0 (s_2 s_1) s_2 (s_4 s_3) s_4}(x_1, x_2, x_3, x_4, x_5; \beta),
\\
L_{s_2 s_1 s_3 s_2 \lambda}(\xx; \beta) & = \G_{w_0 (s_3 s_2 s_1) (s_3 s_2) (s_5 s_4 s_3) (s_5 s_4) s_5}(x_1, x_2, x_3, x_4, x_5, x_6; \beta).
\end{align*}
\end{ex}

The Lascoux polynomials given by Corollary~\ref{cor:vexillary_Lascoux} are not the only ones that are equal to a Grothendieck polynomial.
\ytableausetup{boxsize=0.8em}
For example, if $\lambda = 42 = \ydiagram{4,2}$, then
\[
L_{s_2 \lambda}(\xx; \beta) = \G_{w_0 (s_2 s_4 s_3 s_4)}(x_1, x_2, x_3, x_4, x_5; \beta).
\]
They are however the only Lascoux polynomials equal to a Grothendieck polynomial for which $\lambda$ is a rectangle.

T.~Matsumura and S.~Sugimoto have informed the authors that every flagged Grothendieck polynomial is a Lascoux polynomial by extending the proof of~\cite[Thm.~3.3]{Matsumura19}, which appears in their work~\cite{MS19}. (This is the K-theoretic analog of the fact that every flagged Schur function is a Demazure character~\cite{RS95}.) Thus in particular, for $w$ a vexillary permutation, the Grothendieck polynomial $\G_w$ is known to be a flagged Grothendieck polynomial~\cite[Thm.~5.8]{KMY09}, so $\G_w$ is also the Lascoux polynomial $L_{a}$ for $a$ the Lehmer code of $w$.
In the special case $\beta=0$, A.~Postnikov and R.~Stanley showed that a Demazure character $\pi_w \xx^{\lambda}$ is a flagged Schur function if and only if $w\in \sym_n^{\lambda}$ is $312$-avoiding~\cite{PS09}. These facts motivate the following conjecture.

\begin{conj}\label{conk:312}
Let $\lambda$ be any partition and $w \in \sym_n^{\lambda}$.
Then the Lascoux polynomial $L_{w\lambda}(\xx; \beta)$ is a flagged Grothendieck polynomial if and only if $w$ is $312$-avoiding.
\end{conj}

Consider an arbitrary flagging condition with Lemma~\ref{lemma:flag_condition}. Then we can interpret these Lascoux polynomials as a Jacobi--Trudi-type determinant, where each part is the Segre class of a vector bundle~\cite{HIMN17,HM18}. (Although flagged set-valued tableaux also appear in~\cite{GK15}, that use appears unrelated to ours, as the weights considered in the two contexts seem to be irreconcilable.)

\section{Bijection with K-Kohnert diagrams}
\label{sec:Kohnert}

Recall that there is a natural bijection between the set of semistandard Young tableaux of shape $1^r$ with entries at most $n$ and the collection of subsets of $\{1, \dotsc, n\}$ of size $r$.
For row $i$ (starting from the bottom row and going up) of a K-Kohnert diagram $D$, consider  the subset of $\{1, \dotsc, n\}$ given by the horizontal coordinates of the unmarked boxes. Construct column $i$ (from right to left) of a (\emph{a priori} non-necessarily semistandard) tableau $T$ by applying the natural bijection given above to this subset. Now, for every marked box in position $(x,i)$ of $D$, there is a rightmost unmarked box $(x',i)$ to the left of $(x,i)$. Insert $x$ into the cell of column $i$ containing $x'$.
In other words, we insert $x$ into the topmost (\textit{i.e.} highest) possible cell of column $i$ such that the resulting column is semistandard. (Note that {\it a priori} the rows may not be semistandard.)  Write $\phi(D)$ for the resulting tableau $T$.

It is straightforward to see that the map $\phi$ is invertible and $\beta$-weight preserving. We will show below that $\phi(D)$ is in fact always a semistandard set-valued tableau. Indeed, our main effort in this section is to establish the following.
Recall that $\mcD_{w\lambda}$ is the set of K-Kohnert diagrams obtained from $w\lambda$.

\begin{prop}
\label{prop:K_skyline_bijection}
Let $\lambda = s^r$ be an $r \times s$ rectangle.
For any $w \in \sym_n^{\lambda}$, $\phi$ restricts to a $\beta$-weight preserving bijection
\[
\phi \colon \mcD_{w\lambda} \to \svt^n_w(\lambda).
\]
\end{prop}

\begin{ex}
\label{ex:phi_bijection}
Consider $\lambda = 2^2$ be a $2 \times 2$ square and $w = s_2$.
Under $\phi$ described above, we have
\[
\newcommand{\kohnert}[1]{\begin{tikzpicture}[scale=0.5] #1 \draw[very thin] (0,0) grid (3,2); \end{tikzpicture}}
\ytableausetup{boxsize=1.7em}
\begin{array}{c@{\qquad}c@{\qquad}c@{\qquad}c@{\qquad}c}
\kohnert{ \fill[darkred!60] (0,0) rectangle (1,2); \fill[darkred!60] (2,0) rectangle (3,2); } &
\kohnert{ \fill[darkred!60] (0,0) rectangle (1,2); \fill[darkred!60] (2,0) rectangle (3,1); \fill[darkred!60] (1,1) rectangle (2,2); } &
\kohnert{ \fill[darkred!60] (0,0) rectangle (1,2); \fill[darkred!60] (1,1) rectangle (2,2); \fill[darkred!60] (2,0) rectangle (3,2); \fill[black] (2.5,1.5) circle (0.2); } &
\kohnert{ \fill[darkred!60] (0,0) rectangle (1,2); \fill[darkred!60] (1,0) rectangle (2,2); } &
\kohnert{ \fill[darkred!60] (0,0) rectangle (2,2); \fill[darkred!60] (2,0) rectangle (3,1); \fill[black] (2.5,0.5) circle (0.2); }
\\
\ytableaushort{11,33} &
\ytableaushort{11,23} &
\ytableaushort{11,{2,\!3}3} &
\ytableaushort{11,22} &
\ytableaushort{11,2{2,\!3}}
\end{array}
\]
where we have shaded in the selected boxes and put a $\bullet$ in the marked boxes. Observe that these tableaux are exactly $\svt^n_w(\lambda)$.
\end{ex}

\begin{ex}
We continue Example~\ref{ex:phi_bijection} to $w' = s_1 s_2$ and obtain all of $\svt^3_{w'}(\lambda) =\svt^3(\lambda)$. We show the bijection $\phi$ on the remaining elements:
\[
\newcommand{\kohnert}[1]{\begin{tikzpicture}[scale=0.5] #1 \draw[very thin] (0,0) grid (3,2); \end{tikzpicture}}
\begin{array}{c@{\qquad}c@{\qquad}c@{\qquad}c@{\qquad}c}
\kohnert{ \fill[darkred!60] (1,0) rectangle (3,2); }
&
\kohnert{ \fill[darkred!60] (2,0) rectangle (3,2); \fill[darkred!60] (1,0) rectangle (2,1); \fill[darkred!60] (0,1) rectangle (1,2); }
&
\kohnert{ \fill[darkred!60] (1,0) rectangle (3,2); \fill[darkred!60] (0,1) rectangle (1,2); \fill[black] (1.5,1.5) circle (0.2);}
&
\kohnert{ \fill[darkred!60] (0,1) rectangle (1,2); \fill[darkred!60] (1,0) rectangle (2,2); \fill[darkred!60] (2,0) rectangle (3,1); }
&
\kohnert{ \fill[darkred!60] (0,1) rectangle (3,2); \fill[darkred!60] (1,0) rectangle (3,1); \fill[black] (2.5,1.5) circle (0.2); }
\\
\ytableaushort{22,33} &
\ytableaushort{12,33} &
\ytableaushort{{1,\!2}2,33} &
\ytableaushort{12,23} &
\ytableaushort{12,{2,\!3}3}
\\[22pt]
\kohnert{ \fill[darkred!60] (0,0) rectangle (1,2); \fill[darkred!60] (2,0) rectangle (3,2); \fill[darkred!60] (1,0) rectangle (2,1); \fill[black] (1.5,0.5) circle (0.2); } &
\kohnert{ \fill[darkred!60] (0,0) rectangle (2,2); \fill[darkred!60] (2,0) rectangle (3,1); \fill[black] (1.5,0.5) circle (0.2); } &
\kohnert{ \fill[darkred!60] (0,0) rectangle (3,2); \fill[black] (1.5,0.5) circle (0.2); \fill[black] (2.5,1.5) circle (0.2);}
\\
\ytableaushort{1{1,\!2},33} &
\ytableaushort{1{1,\!2},23} &
\ytableaushort{1{1,\!2},{2,\!3}3}
\end{array}
\]
\end{ex}

In order to prove Proposition~\ref{prop:K_skyline_bijection}, we first construct the equivalent (K-)Kohnert moves on (semistandard) set-valued tableaux.

\begin{dfn}[(K-)Kohnert moves on set-valued tableaux]
\label{def:Kohnert_moves_tableaux}
Let $T \in \svt^n(\lambda)$. Consider an entry $x \in \ZZ$ such that $x \in T$. Let $\mcC$ be the leftmost column of $T$ containing an $x$. Let $\bbb$ be the box in $\mcC$ containing $x$. Let $x'$ be minimal such that $x'+1, x'+2, \dotsc, x \in \mcC$, and let $\bbb'$ be the box in $\mcC$ containing $x'+1$. 

If $x' = 0$ or if $x \neq \min \bbb$ or if $\{x'+1\}, \dotsc, \{x-1\}$ are not in $\mcC$ (\textit{i.e.} if any of $x'+1, x'+2, \dotsc, x-1$ is not the only entry of its box in $\mcC$), then we do not have a \mbox{(K-)Kohnert} move corresponding to $x$. Otherwise define the \defn{Kohnert move} on $T$ to
\begin{enumerate}
\item remove $x$ from $\bbb$;
\item if $x' < x - 1$, then moving all entries $x'+1, \dotsc, x-1$ in $\mcC$ down one row (which in particular moves $x-1$ into $\bbb$); and
\item inserting $x'$ into $\bbb'$.
\end{enumerate}

A \defn{K-Kohnert move} is the same as a Kohnert move except we leave $x \in \bbb$.
\end{dfn}

\begin{lemma}
\label{lem:Kohnert_svt}
Let $T \in \svt^n(\lambda)$, and denote by $T'$ the result of applying any \mbox{(K-)Kohnert} move to $T$. Then $T' \in \svt^n(\lambda)$.
\end{lemma}

\begin{proof}
We consider only the Kohnert move as the K-Kohnert move is similar. We will use the notation from Definition~\ref{def:Kohnert_moves_tableaux}.

The column increasingness condition for $T'$ is satisfied since we only altered $\mcC$ and we took $x'$ to be minimal \textit{i.e.} any entries in or below $\bbb'$ must not contain an $x'$). 
For the row increasingness condition of $T'$, we note that we are decreasing some entries in $\mcC$, so it is sufficient to just consider the entries in the column $\mcC^{\leftarrow}$ directly to the left of $\mcC$. Let $\bbb^{\leftarrow}$ be the box immediately to the left of $\bbb$. By choice of $\mcC$ as leftmost, we must have $x \notin \mcC^{\leftarrow}$ and hence by the row semistandardness of $T$, we must have $\max \bbb^{\leftarrow} < x$. This implies that the box above $\bbb^{\leftarrow}$ must have all entries strictly less than $x - 1$ by column semistandardness of $T$. By iterating this argument, the analogous statement holds for all boxes in $\mcC$ weakly between $\bbb$ and $\bbb'$. Thus, $T' \in \svt^n(\lambda)$.
\end{proof}

Now we prove Proposition~\ref{prop:K_skyline_bijection} by using our flagging characterization of K-Demazure crystals from Lemma~\ref{lemma:flag_condition} and showing that $\phi$ intertwines the (K-)Kohnert moves on K-Kohnert diagrams with the (K-)Kohnert moves on set-valued tableaux.

\begin{proof}[Proof of Proposition~\ref{prop:K_skyline_bijection}]
By Lemma~\ref{lemma:flag_condition}, we have $\svt^n_w(\lambda) = F(\lambda; w)$, so it is sufficient to show $\phi(\mcD_{w\lambda}) = F(\lambda; w)$.

Let $D \in \mcD_{w\lambda}$ and $T = \phi(D)$.
It is straightforward to see that a Kohnert move moving $x$ in column $y$ of $T$ corresponds under $\phi$ to the Kohnert move on $D$ that moves an unmarked box $(x,y)$ to $(x',y)$. Note that taking the leftmost column in  $T$ is equivalent to taking a box in the top of the $x$-th column of $D$. We claim these are all possible Kohnert moves. Indeed, the condition that $x'+1, \dotsc, x-1$ are the only entry in their boxes corresponds to the fact that we did not cross a marked box in $D$. That $x' \geq 1$ is equivalent to the moved box in $D$ staying within $\ZZ_{>0}^2$ after the move. That $x = \min \bbb$ is equivalent to the box at $(x,y)$ in $D$ being unmarked. Hence, we obtain all possible Kohnert moves.
The claim for K-Kohnert moves is similar.

For any $T \in F(\lambda; w)$, any possible Kohnert or K-Kohnert move applied to $T$ yields another element in $F(\lambda; w)$ as entries in a particular row decrease. Thus we have $\phi(\mcD_{w\lambda}) \subseteq F(\lambda; w)$ as the initial skyline diagram corresponds to the tableau $T_{w\lambda}$ with all entries in row $r-k$ being $\{i_k\}$.

To show $\phi(\mcD_{w\lambda}) \supseteq F(\lambda; w)$, we will show that every set-valued tableau in $F(\lambda; w)$ can be obtained from $T_{w\lambda}$ by applying (K-)Kohnert moves. For a given flagged set-valued tableau $T \in F(\lambda; w)$, we start by applying (K-)Kohnert moves on the upper left box of $T_{w\lambda}$ until we get the entry in the upper left box of $T$. Recall that the corresponding (K-)Kohnert move acts on this entry/column since it acts on the leftmost applicable column. We repeat this process moving across the first \textit{i.e.} topmost) row. As such, there will be no interactions between the different Kohnert moves. We then repeat this for the second row, then the third, and so on until all entries have been changed to $T$. All the tableaux produced along the way lie in $\svt^n(\lambda)$ by Lemma~\ref{lem:Kohnert_svt}.

We claim this procedure always generates a sequence of (K-)Kohnert moves in $\mcD_{w\lambda}$. Indeed, what we are doing is moving the first column from the skyline diagram to its appropriate spots for the final K-Kohnert diagram $\phi^{-1}(T)$ starting from the top. It is easy to see that by doing this process, we are never moving a box across another box.\footnote{
One can also see this process as a set-valued variation of the Kohnert tableaux of~\cite{AS18:Kohnert} by doing the same labeling procedure on unmarked boxes and labeling any marked box $\bbb$ with an $x'$, where $x$ is the label of the rightmost unmarked box to the left of $\bbb$.
}
Hence, such (K-)Kohnert moves are always valid.
Thus, $\phi(\mcD_{w\lambda}) = F(\lambda; w)$.
\end{proof}

\begin{ex}
Let $\lambda = 3^2$ be a $2 \times 3$ rectangle and consider $n = 4$. We exhibit the sequence of (K-)Kohnert moves described in the proof of Proposition~\ref{prop:K_skyline_bijection} to obtain the element $\scalebox{0.7}{$\ytableaushort{1{1,\!2}2,{2,\!3}3{3,\!4}}$} \in \svt^4_{s_1 s_3 s_2} (\lambda)$ from the initial tableau $\scalebox{0.7}{$\ytableaushort{222,444}$}$\,:
\[
\newcommand{\kohnert}[1]{\begin{tikzpicture}[scale=0.4,baseline=13] #1 \draw[very thin] (0,0) grid (4,3); \end{tikzpicture}}
\newcommand{\regar}{\xleftarrow[\hspace{13pt}]{}}
\newcommand{\Kar}{\xleftarrow[\hspace{13pt}]{K}}
\ytableausetup{boxsize=1.7em}
\begin{array}{c@{\;}c@{\;}c@{\;}c@{\;}c@{\;}c@{\;}c@{}c}
\ytableaushort{1{1,\!2}2,{2,\!3}3{3,\!4}} & \Kar &
\ytableaushort{1{1,\!2}2,{2,\!3}34} & \regar &
\ytableaushort{1{1,\!2}2,{2,\!3}44} & \Kar &
\ytableaushort{1{1,\!2}2,344}
& \; \regar
\\[18pt]
\kohnert{ \fill[darkred!60] (0,1) rectangle (3,3); \fill[darkred!60] (1,0) rectangle (4,1); \fill[black] (1.5,1.5) circle (0.2); \fill[black] (2.5,2.5) circle (0.2); \fill[black] (3.5,0.5) circle (0.2);} & \Kar &
\kohnert{ \fill[darkred!60] (0,1) rectangle (3,3); \fill[darkred!60] (0,1) rectangle (1,3); \fill[black] (1.5,1.5) circle (0.2); \fill[darkred!60] (1,0) rectangle (2,1); \fill[darkred!60] (3,0) rectangle (4,1); \fill[black] (2.5,2.5) circle (0.2); } & \regar &
\kohnert{ \fill[darkred!60] (1,0) rectangle (2,2); \fill[darkred!60] (0,1) rectangle (1,3); \fill[black] (1.5,1.5) circle (0.2); \fill[darkred!60] (3,0) rectangle (4,2); \fill[darkred!60] (1,2) rectangle (3,3); \fill[black] (2.5,2.5) circle (0.2); } & \Kar &
\kohnert{ \fill[darkred!60] (1,0) rectangle (2,2); \fill[darkred!60] (0,1) rectangle (1,3); \fill[black] (1.5,1.5) circle (0.2); \fill[darkred!60] (3,0) rectangle (4,2); \fill[darkred!60] (2,2) rectangle (3,3); }
& \;\regar
\\[20pt]
& \regar &
\ytableaushort{1{1,\!2}2,444} & \Kar &
\ytableaushort{122,444} & \regar &
\ytableaushort{222,444}
\\[18pt]
& \regar &
\kohnert{ \fill[darkred!60] (1,0) rectangle (2,2); \fill[darkred!60] (0,1) rectangle (1,3); \fill[black] (1.5,1.5) circle (0.2); \fill[darkred!60] (3,0) rectangle (4,3); } & \Kar &
\kohnert{ \fill[darkred!60] (1,0) rectangle (2,2); \fill[darkred!60] (0,2) rectangle (1,3); \fill[darkred!60] (3,0) rectangle (4,3); } & \regar &
\kohnert{ \fill[darkred!60] (1,0) rectangle (2,3); \fill[darkred!60] (3,0) rectangle (4,3); }
\end{array}
\]
where the diagrams under $\phi^{-1}$ are given below the set-valued tableaux.
\end{ex}

\begin{remark}
\label{rem:Key_property}
We note that the proof of the intertwining of (K-)Kohnert moves did not require $\lambda$ to be a rectangle. However, that $\phi$ is a bijection does require that $\lambda$ is a rectangle as otherwise the image of the skyline diagram will not be a partition. For example,
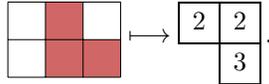
\[
\ytableausetup{boxsize=1.5em}
\begin{tikzpicture}[scale=0.5,baseline=13]
\fill[darkred!60] (1,0) rectangle (2,2); \fill[darkred!60] (2,0) rectangle (3,1);
\draw[very thin] (0,0) grid (3,2);
\end{tikzpicture}
\quad \longmapsto \quad
\ytableaushort{22,\none3}\,.
\]
\end{remark}

\begin{thm}
For $\lambda = s^r$ an $r \times s$ rectangle, we have
\[
L_a(\xx; \beta) = \sum_{D \in \mcD_a}  \wt_{\beta}(D);
\]
Hence, the Ross--Yong--Kirillov Conjecture (Conjecture~\ref{conj:K_skyline}) holds for $L_{a}$ when $a$ is any weak composition with a unique nonzero part size.
\end{thm}

\begin{proof}
This follows from the definition of a K-crystal together with Lemma~\ref{lemma:flag_condition}, Theorem~\ref{thm:Krystal_rectangles}, and Proposition~\ref{prop:K_skyline_bijection}.
\end{proof}

To the best of our knowledge, no other cases of Conjecture~\ref{conj:K_skyline} have been previously established.

\section{Bijection with set-valued skyline tableaux}
\label{sec:skyline_tableaux}

Consider a partition $\lambda$ and permutation $w$.
Define
\begin{equation}
\label{eq:svt_demazure_atom}
\overline{\svt}^n_w(\lambda) := \svt^n_w(\lambda) \setminus \bigcup_{v < w} \svt^n_v(\lambda),
\end{equation}
where the union is taken over all $v$ strictly less than $w$ in Bruhat order.
We have
\[
\overline{\varpi}_w \xx^{\lambda} = \sum_{v \leq w} (-1)^{\ell(w) - \ell(v)} \varpi_v \xx^{\lambda}
\]
by applying M\"obius inversion on Bruhat order and Equation~\eqref{eq:lascoux_to_atoms}.
Therefore, Conjecture~\ref{conj:skyline_tableaux} is equivalent by inclusion-exclusion to showing that
\begin{equation}
\label{eq:equivalent_conjecture}
\overline{L}_{w\lambda} = \ch_{\beta} \left( \overline{\svt}^n_w(\lambda) \right).
\end{equation}

\begin{prop}
\label{prop:skyline_tableaux_bijection}
Let $\lambda = s^r$ be an $r \times s$ rectangle.
For any $w \in \sym_n^{\lambda}$, there exists a $\beta$-weight preserving bijection
\[
\psi \colon \skyline_{w\lambda} \to \overline{\svt}^n_w(\lambda).
\]
\end{prop}

\begin{proof}
Let $i_k < \cdots < i_0$ be from the reduced expression of $w$ given by Lemma~\ref{lemma:coset_repr}.
By Lemma~\ref{lemma:flag_condition}, we can characterize $\overline{\svt}^n_w(\lambda)$ by $T \in \overline{\svt}^n_w(\lambda)$ if and only if, for all $0 \leq k < r$, row $k$ of $T$ has at least one $i_k + 1$ in it and no number strictly greater that $i_k+1$. By the fact that the rows of $T$ are weakly increasing, an equivalent characterization is that the largest entry of the rightmost box in row $k$ of $T$ must be $i_k+1$.

Define a map $\psi \colon \skyline_{w\lambda} \to \overline{\svt}^n_w(\lambda)$ as follows. Consider some $S \in \skyline_{w\lambda}$ and define $T := \psi(S)$ by
\begin{enumerate}
\item sorting the anchor entries in each row in increasing order left to right and join the columns together;
\item placing each free entry $f$ in the leftmost box of its row such that $f$ is less than the anchor entry (\textit{i.e.} so that the row is strictly increasing but the free entries remain free);
\item take the transpose of the result from the previous step; that is construct the $i$-th column of $T$ from the $(r+1-i)$-th row as in Section~\ref{sec:Kohnert}.
\end{enumerate}
To see that $\psi$ is injective and that its image is contained in $\svt^n(\lambda)$, we note that $\psi$ is a restriction of the bijection $\hat{\rho}$ of~\cite[Thm.~2.4]{Monical16} from set-valued skyline tableaux to \emph{reverse} semistandard set-valued tableaux, except that we have reversed the rows and columns of the image tableaux. Indeed, reversing the rows and the columns for rectangular shapes is clearly a bijection from reverse semistandard set-valued tableaux to (ordinary) semistandard set-valued tableaux.

To see that the image of $\psi$ is $\overline{\svt}^n_w(\lambda)$, we note that the anchors of the first row of $S$ are already in increasing order. Thus, the anchor of the $k$-th column of $S$ under $\psi$ is the largest entry in the $k$-th row of $\psi(S)$. By construction, the largest entry of the $k$-th column of $S$ is $i_k + 1$, and hence, $\psi$ is surjective.

Finally, it is clear that $\psi$ is $\beta$-weight preserving, so $\psi$ is the desired bijection.
\end{proof}

\begin{ex}
Let $\lambda = 2^2$ be a $2 \times 2$ rectangle and $n = 3$. Then the set-valued skyline tableaux $\skyline_{s_2\lambda}$ and their corresponding element in $\overline{\svt}^3_{s_2}(\lambda)$ under $\psi$ are given by
\begin{align*}
\ytableausetup{boxsize=1.7em}
\begin{ytableau}
1 & \none & 3 \\
1 & \none & 3
\end{ytableau}
& \quad \longmapsto \quad 
\ytableaushort{11,33}\,,
& 
\begin{ytableau}
1 & \none & 2 \\
1 & \none & 3
\end{ytableau}
& \quad \longmapsto \quad 
\ytableaushort{11,23}\,,
\\
\begin{ytableau}
1 & \none & 2,\!3 \\
1 & \none & 3
\end{ytableau}
& \quad \longmapsto \quad 
\ytableaushort{11,{2,\!3}3}\,,
&
\begin{ytableau}
1 & \none & 2 \\
1 & \none & 2,\!3
\end{ytableau}
& \quad \longmapsto \quad 
\ytableaushort{11,2{2,\!3}}\,.
\end{align*}
\end{ex}

\begin{ex}
Let $\lambda = 4^3$ be a $4 \times 3$ rectangle, $n = 10$, and
\[
w = (s_4 s_3 s_2) (s_5 s_4 s_3) (s_8 s_7 s_6 s_5 s_4).
\]
Consider the following is a set-valued skyline tableau in $\skyline_{w\lambda}$
\[
\ytableausetup{boxsize=1.7em}
S := 
\begin{ytableau}
1 & \none & \none & \none & 3 & 4 & \none & \none & 2 \\
1 & \none & \none & \none & 3 & {4,\!6} & \none & \none & {2,\!7} \\
1 & \none & \none & \none & {4,\!5} & 6 & \none & \none & {7,\!9} \\
\end{ytableau}
\]
Applying the steps of the bijection $\psi$, we obtain
\[
S \; \longmapsto \;
\ytableaushort{1234,1367,1569}
\; \longmapsto \;
\ytableaushort{1234,1{2,\!3}{4,\!6}7,1{4,\!5}6{7,\!9}}
\; \longmapsto \;
\ytableaushort{111,2{2,\!3}{4,\!5},3{4,\!6}6,47{7,\!9}} \in \overline{\svt}^n_w(\lambda).
\]
\end{ex}

%

To the best of our knowledge, the following theorem is the first to prove any case of Conjecture~\ref{conj:skyline_tableaux}.
\begin{thm}
For $\lambda = s^r$ an $r \times s$ rectangle, we have
\[
\overline{L}_{w\lambda} = \sum_{S \in \skyline_{w\lambda}} \wt_{\beta}(S).
\]
Hence, Monical's Skyline Conjecture (Conjecture~\ref{conj:skyline_tableaux}) holds for $L_{a}$ when $a$ is any weak composition with a unique nonzero part size.   
\end{thm}

\begin{proof}
This follows from the definition of a K-crystal together with Theorem~\ref{thm:Krystal_rectangles}, Equation~\eqref{eq:equivalent_conjecture}, and Proposition~\ref{prop:skyline_tableaux_bijection}.
\end{proof}

\section{K-key tableaux}
\label{sec:Key_tableaux}

A \defn{key tableau} $K$ is a semistandard tableau such that the entries in the $j$-th column of $K$ are a subset of those in the $(j-1)$-st column of $K$.
One method to compute a Demazure character $\kappa_{w\lambda}$ is by summing over all semistandard tableaux of shape $\lambda$ whose (right) key entries are less than corresponding entry in the unique key tableaux $K_{w\lambda}$ of weight $w\lambda$~\cite{LS90}. (This fact explains why Demazure characters are also known as key polynomials.)

Furthermore, every semistandard tableau $T$ has a unique \defn{(right) key tableau} $k(T)$ associated with it (we refer the reader to~\cite{Willis13} for an algorithm), and a Demazure atom can be computed as a generating function for all semistandard tableaux $T$ with $k(T) = K_{w\lambda}$~\cite{LS90}. (See~\cite{PW15} for much further discussion of these (and related) formulas.)
Let $\prec$ denote the partial order on semistandard tableaux of shape $\lambda$ such that $T \preceq T'$ if and only if every entry of $T$ is at most the corresponding entry in $T'$.

Based on the bijection from Proposition~\ref{prop:skyline_tableaux_bijection} and the (K-)Kohnert moves on set-valued tableaux (Definition~\ref{def:Kohnert_moves_tableaux}), the following is a natural possible extension of key tableaux to the K-theory setting.
For $T \in \svt^n(\lambda)$, define $\mcK(T) := k\bigl( \max(T) \bigr)$, where $\max(T)$ is semistandard tableau obtained by taking the greatest entry in each box of $T$.
Thus Theorem~\ref{thm:Krystal_rectangles} and Lemma~\ref{lemma:flag_condition} imply that for $\lambda = r^s$
\begin{equation}
\label{eq:rectangle_lascoux}
L_{w\lambda}(\xx; \beta) = \sum_{\substack{T \in \svt^n(\lambda) \\ \mcK(T) \preceq K_{w\lambda}}} \wt_{\beta}(T),
\qquad\qquad
\overline{L}_{w\lambda}(\xx; \beta) = \sum_{\substack{T \in \svt^n(\lambda) \\ \mcK(T) = K_{w\lambda}}} \wt_{\beta}(T),
\end{equation}
or equivalently summed over $\svt^n_w(\lambda)$ and $\overline{\svt}^n_w(\lambda)$ respectively.
However, these formulas do not work for general $\lambda$ as, for example,
\[
\ytableausetup{boxsize=2.1em}
\mcK\left( \ytableaushort{{1}{1,\!2,\!3},{2,\!3}} \right) = \ytableaushort{13,3}\,, \ytableausetup{boxsize=1.5em}
\] 
but it can only contribute to the Lascoux polynomial/atom corresponding to $w_0 \lambda$, where $\lambda = 21$, as it has an excess of $3$.
Moreover, the weak K-crystal in~\cite[Fig.~7]{MPS18} does \emph{not} decompose the K-crystal into atoms as given by Equation~\eqref{eq:svt_demazure_atom} as \scalebox{0.7}{$\ytableaushort{1{2,\!3},3}$} should not be in the atom corresponding to $w_0$.

Instead, we conjecture that Equation~\eqref{eq:rectangle_lascoux} should be modified by using the Lusztig involution to obtain a combinatorial interpretation of general Lascoux polynomials and atoms.
Recall that the Lusztig involution on the highest weight crystal $B(\mu)$ is defined by sending the highest weight element $U$ to the lowest weight element $U^*$ and extended to the remaining elements in $B(\mu)$ by
\begin{equation}
\label{eq:lusztig_involution}
(f_{n-i} T)^* = e_i(T^*),
\qquad
(e_{n-i} T)^* = f_i(T^*),
\qquad
\wt(T^*) = w_0 \wt(T).
\end{equation}
We can extend this naively to $\svt^n(\lambda)$ by acting on each irreducible component $B(\mu)$.
Define the \defn{(right) K-key tableau} of a set-valued tableau $T \in \svt^n(\lambda)$ by
\[
K(T) := k(\min(T^*)^*),
\]
where $\min(T)$ is the semistandard tableau obtained from $T$ by taking the least entry in each box of $T$.

\begin{conj}
\label{conj:Key}
Let $\lambda$ be a partition.
Define the sets
\begin{align*}
\svt^n_w(\lambda) & := \{ T \in \svt^n(\lambda) \mid K(T) \preceq K_{w\lambda} \},
\\ \overline{\svt}^n_w(\lambda) & := \{ T \in \svt^n(\lambda) \mid K(T) = K_{w\lambda} \}.
\end{align*}
Then we have
\[
L_{w\lambda}(\xx; \beta) = \sum_{T \in \svt^n_w(\lambda)} \wt_{\beta}(T),
\qquad\qquad
\overline{L}_{w\lambda}(\xx; \beta) = \sum_{T \in \overline{\svt}^n_w(\lambda)} \wt_{\beta}(T),
\]
\end{conj}

Although at first glance Conjecture~\ref{conj:Key} looks rather different from our proved formulas in the rectangular cases, we now will show that Equation~\eqref{eq:rectangle_lascoux} establishes Conjecture~\ref{conj:Key} when $\lambda$ is a rectangle.\footnote{Conjecture~\ref{conj:Key} has now been proven in~\cite{BSW20}.}
To do so, we will construct a \defn{K-Lusztig involution}
\[\star \colon \svt^n(\lambda) \to \svt^n(\lambda)\]
 that also satisfies Equation~\eqref{eq:lusztig_involution}, but is a twist of the Lusztig involution by an automorphism of the $U_q(\fsl_n)$-crystal $\svt^n(\lambda)$ (\textit{i.e.} it nontrivially permutes the irreducible components).
Let $\lambda = r^s$ be a rectangle and $T \in \svt^n(\lambda)$. Define $T^{\star}$ to be the set-valued tableau obtained by rotating the tableau $180^{\circ}$ and then replacing each $i \mapsto n+1-i$.
We note this is a well-known description of the Lusztig involution (also known as the Sch\"utzenberger involution or evacuation~\cite{Lenart07}) on semistandard tableaux of shape $\lambda$.

\begin{prop}
\label{prop:Lusztig_reversal}
Let $\lambda$ be a rectangle.
The K-Lusztig involution $\star$ satisfies Equation~\eqref{eq:lusztig_involution}.
Moreover, for $T \in \svt^n(\lambda)$ as a tensor product of rows $T = R_1 \otimes \cdots \otimes R_k$, we have
\[
T^{\star} = R_k^* \otimes \cdots \otimes R_1^*.
\]
\end{prop}

\begin{proof}
The first claim follows from the definition of the crystal operators. We leave the details to the reader.
For the second claim, we first note that $T^{\star} = R_k^{\star} \otimes \cdots \otimes R_1^{\star}$, and so it is sufficient to show $R_1^{\star} = R_1^*$.
This follows from a straightforward induction on depth (\textit{i.e.}, the number of crystal operators applied from the highest weight element) and Equation~\eqref{eq:lusztig_involution}.
\end{proof}

Proposition~\ref{prop:Lusztig_reversal} also suggests that Conjecture~\ref{conj:Key} holds for a definition of a (right) K-key tableau by
\[
K'(T) := k(\min(T^{\dagger})^*),
\]
where $T^{\dagger}$ is constructed from $T$ according to \emph{any} automorphism of $\svt^n(\lambda)$ such that $\wt(T^{\dagger}) = w_0 \wt(T)$.
However, given a (weak) K-crystal structure on $\svt^n(\lambda)$, it would be preferable to have a $T^{\dagger}$ construction that matches the labeling of tableaux $T$ by K-keys $K'(T)$ with the decomposition of the K-crystal by K-Demazure subcrystals, as is the case with our K-Lusztig involution $T^\star$.
Furthermore, it is likely that in general we want $T^{\star} = R_k^* \otimes \cdots \otimes R_1^*$ as in Proposition~\ref{prop:Lusztig_reversal}, but this would require an appropriate K-rectification or insertion scheme in order to obtain a result back in $\svt^n(\lambda)$.

We also believe there exists an insertion scheme analogous to the one given by S.~Mason in~\cite[Sec.~3.3]{Mason08} to construct a bijection between $\overline{\svt}_w^n(\lambda)$ and $\skyline_{w\lambda}$. This will possibly be a variant of the insertion given in~\cite{Buch02} similar to how Mason's map is a variant of the classical RSK algorithm.

\bibliographystyle{alpha}
\bibliography{crystals}{}
\end{document}